\documentclass[a4paper,11pt]{article}
\title{}
\author{}
\usepackage{float}
\usepackage{graphicx} 
\usepackage{amsthm}
\usepackage{amsmath}
\usepackage{amsfonts} 
\usepackage{amsopn}
\usepackage{relsize}
\usepackage{enumerate}
\usepackage{amssymb}
\usepackage{tikz}
\usetikzlibrary{arrows}
\usepackage{scalefnt}
 
\newcommand\dist[2]{d\left(#1, #2\right)}

\DeclareMathOperator*{\centroid}{\mathlarger{\Conv}}
\newtheorem{theorem}{Theorem}[section]
\newtheorem{definition}{Definition}[section]

\newtheorem{lemma}{Lemma}[section]
\newtheorem{proposition}{Proposition}[section]

\newcommand{\Conv}{\mathop{\scalebox{1.5}{\raisebox{-0.2ex}{$\ast$}}}}%
\begin{document}

\title{The Existence of a  Billiard Orbit in the Regular Hyperbolic Simplex}
\date{}
\author{Oded Badt, \ \ Yaron Ostrover}
\maketitle
\begin{abstract}
In this note we  establish the existence of a $n+1$ periodic billiard trajectory inside an $n$-dimensional regular 
simplex in the hyperbolic space, which hits  the interior of every facet exactly once. 
\end{abstract}
\section{Introduction and Result}

The billiard dynamical system describes the motion of a massless particle in a
domain with a perfectly reflecting boundary (see e.g.~\cite{KT,Tab}  for two
excellent surveys on the subject). A particular intriguing class of examples,
whose dynamics is in general neither integrable nor chaotic, is the class of
polygonal billiards. On the one hand, this class serves as a promising model for
quantum chaos~\cite{BR}, and on the other  it is closely related  to geodesic flows
on flat surfaces and Teichm\"uller  dynamics~\cite{MT}.

In 1775, J.F. de Tuschis a Fagnano observed that in every acute triangle in the
Euclidean plane, the orthic triangle, whose vertices are the feet of the
altitudes, represents a periodic billiard trajectory. Nevertheless, the
existence of periodic billiard orbits for polygonal billiards turns out to be a
challenging  question, even for seemingly simple examples like obtuse triangles
in the plane (see~\cite{VorGalSte,Gut}). Nearly nothing is known in higher
dimensions.

Recently, using barycenter coordinates, the existence of a Fagnano periodic
billiard trajectory inside the regular simplex in the Euclidean space
${\mathbb E}^n$  was established in~\cite{BeRa}.  Despite the lack of linear
structure, 
in this note we extend the result of~\cite{BeRa} to the hyperbolic
space $\mathbb{H}^n$. More precisely, we consider hyperbolic regular simplices
i.e., the convex hulls of $n + 1$ points (vertices) in the hyperbolic space
for which all the distances between two distinct vertices are equal. The
billiards dynamics inside a hyperbolic simplex is defined in much the same way
as in the Euclidean case: the particle moves along geodesic arcs 
 within the simplex, interrupted by elastic collisions against the boundary where
the motion undergoes a specular reflection (see Section~\ref{sec-Background} below 
for the precise definition).

Our main result in this note is the following:

\begin{theorem} \label{main-thm} Let $\triangle^n$  be a regular $n$-simplex with a
given edge length in the hyperbolic space  ${\mathbb H}^n$. Then, there exists an
$(n+1)$-periodic billiard trajectory inside $\triangle^n$  which hits  the interior of
every facet exactly once.  \end{theorem}

The proof of Theorem~\ref{main-thm}  follows the steps of~\cite{BeRa},  
where the main new input being the approach by which we  overcome several difficulties arising from 
the lack on linear structure in the hyperbolic space.  Moreover, the same approach can be used to 
obtain the  hyperbolic analog of the  $2n$-periodic orbit constructed in~\cite{BeRa}.
The details are spelled out in~\cite{OB}.

\noindent{\bf Remark I:}  {\rm  In the case where $n=2$, the regular simplex is
an equilateral triangle in the hyperbolic plane and the billiard orbit
provided by the theorem above  coincides with the orthic triangle i.e.,  the
triangle whose vertices are the endpoints of the altitudes, which is the  well
known Fagnano billiard trajectory.   In contrast with the two-dimensional
case, for $n>2$, a direct computation shows that as in the Euclidean case, the
trajectory  connecting the midpoint of the facets, which in  the regular
simplex coincides with the trajectory connecting the endpoints of the
altitudes,   fails to form a billiard orbit.  We do not know the precise
geometric (or physical) meaning of the bouncing points of the billiard orbit
provided by the theorem above.   }

\noindent{\bf Remark II:}  {\rm On top of the theoretical mathematical interest in studying billiard dynamics in the framework of hyperbolic geometry, it also has several implications to physics. As an example, we mention the remarkable connection between certain polyhedral billiards in the hyperbolic space and solutions to the vacuum Einstein equations in the vicinity of a space-like singularity, which was uncovered in a series of works starting with the pioneering papers of Belinskii, Khalatnikov and Lifshitz (see e.g.,~\cite{DH,DHN,DHN1} and the references therein). }

\noindent{\bf Structure of the paper:} In Section~\ref{sec-Background} we
recall some relevant facts from hyperbolic geometry, and introduce some of the
technical ingredients needed later in the proof of Theorem~\ref{main-thm},
which in turn is given in  Section~\ref{section-proof-main-theorem}.

\noindent{\bf Acknowledgement:} 
The second named author was partially supported by a Reintegration Grant SSGHD-268274 within the 7th European
community framework programme, and by the ISF grant No. 1057/10. 

\section{Background from Hyperbolic Geometry} \label{sec-Background}

In this section we first recall some relevant notions and facts from
hyperbolic geometry.  For a detailed exposition of the subject, see e.g, the
books~\cite{AleVinbSolo}. Then, in Subsection~\ref{section-regular-simplices},
we provide  the main ingredients in the proof of Theorem~\ref{main-thm} above.

The $n$-dimensional hyperbolic space ${\mathbb H}^n$ is the unique simply
connected and complete $n$-dimensional  Riemannian manifold of constant
curvature $-1$.    In what follows we shall denote by $d$ the corresponding
hyperbolic metric. Among the several models for the hyperbolic space, one can
consider the half-space conformal model (also denoted by ${\mathbb H}^n$ to
simplify notation) given by the metric space \begin{align*} {\mathbb H}^n =
\left ( {\mathbb E}^n_{+}, \ ds^2 = {\frac {dx_1^2 + \cdots+dx_n^2} { x_1^2 }}
\right), \end{align*} where ${\mathbb E}^n_{+} =  \{ (x_1,\ldots,x_n) \in
{\mathbb E}^n \, | \, x_1 >0 \}$ is the upper half space of the Euclidean
space ${\mathbb E}^n$.

The compactification $\overline {\mathbb H}^n = {\mathbb H}^n \cup \partial
{\mathbb H}^n$ consists of ${\mathbb H}^n$ together with the set  $\partial
{\mathbb H}^n =  {\mathbb E}^{n-1} \cup \{ \infty  \}$ of its points at
infinity.  It is well known that any two points $A,B \in {\mathbb H}^n$ can be
joined by a unique  geodesic (h-line), we shall denote the geodesics segment connecting them
by $[AB]$.
A hyperplane in ${\mathbb H}^n$ is a codimension-one totally geodesic subspace of ${\mathbb H}^n$, 
which divides 
the hyperbolic space into two half-spaces (see e.g.,~\cite{AleVinbSolo}, Chapter 1, \textsection 3).
 Note that any hyperplane is isometric to ${\mathbb H}^{n-1}$. For example, 
in the half space model mentioned above, the geodesic hyperplanes are $(n-1)$-spheres and $(n-1)$-planes orthogonal to $\partial {\mathbb H}^n$.
As in the Euclidean case, a reflection in the hyperbolic space ${\mathbb H}^n$ with respect to a given hyperplane is an isometric involution 
fixing point-wise  the hyperplane and isometrically interchanging the two half-spaces of ${\mathbb H}^n$ associated with it.
It is well known that there is a unique reflection in
every hyperplane in ${\mathbb H}^n$.

Finally, a set $ X \subset {\mathbb H}^n$ is said to be {\it convex}  if for any
two points $A,B \in X$ it contains the segment $[AB]$. The convex hull of a
set $Y \subset {\mathbb H}^n$ is the intersection of all the convex sets in
${\mathbb H}^n$ containing  $Y$.

\subsection{Hyperbolic center of mass}

Following~\cite{Galp,Stahl}, we define the notion of center of mass in
hyperbolic space. A point mass is an ordered pair $\left(X,x\right)$, where its
location $X \in {\mathbb H}^n$ is a point of the hyperbolic space and its weight
$x$ is a non-negative real number. \begin{definition}[\cite{Galp,Stahl}] \label
{def-centroid} Given any two point-masses $\left(X, x\right)$ and $\left(Y,
y\right)$, their center of mass, or $centroid$, $\left(X, x\right) \Conv \left(Y,
y\right)$ is the  point mass $\left(Z, z\right)$, such that $Z$ is the unique
point that lies on the segment $[XY]$ and satisfies
\begin{align*}
  x\sinh d(X,Z) = y\sinh d(Y,Z).
\end{align*}
Its corresponding mass is given by
\begin{align*}
 z = x\cosh{d\left(X, Z\right)} + y\cosh{d\left(Y, Z\right)}.
\end{align*}
\end{definition}

As shown in~\cite{Galp, Stahl}, the  operator $\Conv$ is well defined, commutative
and associative. This allows, in particular, to define the {\it centroid} of a
finite set of point masses. Moreover, it follows immediately from Definition~\ref{def-centroid}
that $$(X, w_1) \Conv (X, w_2) = (X, w_1 + w_2), \  \ \  \, (X, w) \Conv (Y,0) = (X,w),$$
and that for every non-negative real number $\lambda$ one has,
\begin{equation*} \label{linearity-of-centroid} 
(X, x) \Conv (Y, y) = (Z, z) \, \Leftrightarrow \,  (X, \lambda x) \Conv (Y, \lambda
y) = (Z, \lambda z). \end{equation*}
Furthermore, since the  center of mass is defined solely by
means of geodesics and distances along them, it commutes with isometries. 
More precisely, for every isometry $\sigma$ of the hyperbolic space ${\mathbb H}^n$ one has
 $$(X, x) \Conv (Y, y) = (Z, z) \,
\Leftrightarrow \, (\sigma X, x) \Conv (\sigma Y, y) = (\sigma Z, z).$$

\subsection{Regular simplices in $\mathbb{H}^n$} \label{section-regular-simplices}

In this subsection we introduce some facts regarding regular simplices in the hyperbolic space. 
In particular, we compute the centroid of the regular simplex (with unit mass on the vertices), and prove
Proposition~\ref{corollary-combined-mass-equation}, which plays a key role in the proof of Theorem~\ref{main-thm}.
We start with the following:

\begin{definition} An $n$-simplex $\triangle^n$ in ${\mathbb H}^n$ is the
convex hull of $n+1$ points in ${\mathbb H}^n$, called vertices.  It is said to
be regular if every permutation of its vertices is induced by an isometry of
${\mathbb H}^n$.
\end{definition}

It is well known (see e.g.~\cite{AleVinbSolo}, Chapter 6, \textsection 2), that up to isometries of $\mathbb{H}^n$, 
for every $a\in\mathbb{R}_+$ there is a unique hyperbolic regular $n$-simplex 
in $\mathbb{H}^n$ with edge length $a$. In what follows we shall denoted this simplex by
$\triangle^n_a$.

\begin{definition} \label{regular-simplex-midpoint-def} A point $C\in{\mathbb
H}^n$ that is equidistant from all the vertices of the simplex $\triangle_a^n$
is called a {\it midpoint} of $\triangle_a^n$. \end{definition} It is not hard
to check that for any $n \in {\mathbb N}$ and $a >0$, there is a unique
midpoint $C_a^n$  of $\triangle_a^n$.

\begin{definition} \label{facet-definiton} Let $\triangle^n \subset {\mathbb H}^n$ be a
regular hyperbolic $n$-simplex with vertices $\{V_0,\ldots,V_n\}$.
For every $0 \leq j \leq n$, the $j$-facet of  $\triangle^n$, denoted by $F_j$,  is the 
regular $(n-1)$-simplex given by the convex hull of the vertices  $\left\{V_k\right\}$, where $0 \le k \le
n$, and $k \neq j$. In what follows we shall denote by $\sigma_j$ the reflection in ${\mathbb H}^n$ with respect
to the (unique) hyperplane in which $F_j$ lies. 
\end{definition}

In the following proposition we gather some basic properties of the regular
simplex in the hyperbolic space (see Figure~\ref{fig:right-angle-triangle-in-simplex} below).

\begin{proposition} \label{prop-about-distances-midpoint-vertices}
\label{hyperbolic-regular-simplex-measurements}
Let $\triangle^n_a$ be a hyperbolic regular $n$-simplex with edge length $a \in {\mathbb R}^+$.
Let $\left\{V_0,\ldots,V_n\right\}$ be its vertices, $C_n$ its midpoint, and
$C_{n-1}$ the midpoint of the facet $F_n$ defined by
$\left\{V_0,\ldots,V_{n-1}\right\}$. Then for any $n \geq 1$, 
\begin{itemize}
\item[(i)] $C_n\in\left[C_{n-1}V_n\right]$, \vskip4pt \label{midpoint-on-keel}
\item[(ii)] $\angle C_n C_{n-1}V_j = \angle V_nC_{n-1}V_j = \frac{\pi}{2}, \ {\rm for \ every \ } 0 \le j \le n \ {\rm and} \ n>1, $  \vskip4pt \label{right-angle-in-simplex}
\item[(iii)] $\cosh^2{d\left(V_j, C_n\right)} = \frac{n\cosh{a} + 1}{n + 1}, \ {\rm for \ every \ } 0 \le j \le n, $ \vskip4pt \label{vertex-to-midpoint-distance}
\item[(iv)] $\cosh^2{d\left(V_n, C_{n-1}\right)} = \frac{n\cosh^2{a}}{\left(n-1\right)\cosh{a} + 1}.$ \label{vertex-to-facet-midpoint-distance}
\end{itemize} 
\end{proposition}

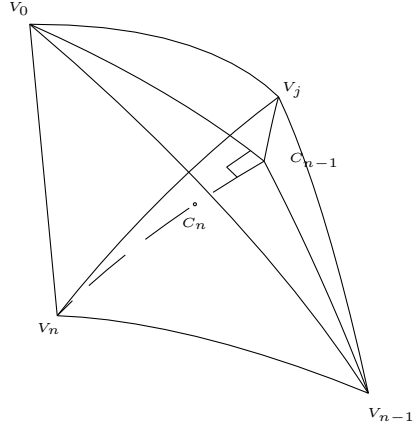
\begin{figure}[H] \centering \begin{tikzpicture}
\draw(-1.06,-2.136)--(-1.421,1.729); \draw(1.852,0.765)--(1.673,0.924)--(1.468
,1.073)--(1.235,1.211)--(0.97,1.337)--(0.673,1.449)--(0.339,1.546)--(-0.034,1.
624)--(-0.45,1.683)--(-0.911,1.718)--(-1.421,1.729)--(-1.421,1.729); \draw(1.8
52,0.765)--(1.566,0.546)--(1.274,0.307)--(0.978,0.049)--(0.678,-0.226)--(0.378
,-0.516)--(0.079,-0.821)--(-0.216,-1.137)--(-0.506,-1.464)--(-0.788,-1.797)--(
-1.06,-2.136); \draw(-1.06,-2.136)--(-0.739,-2.155)--(-0.395,-2.191)--(-0.028,
-2.246)--(0.359,-2.319)--(0.767,-2.411)--(1.193,-2.523)--(1.635,-2.655)--(2.09
,-2.805)--(2.557,-2.975)--(3.032,-3.164); \draw(-1.421,1.729)--(-0.831,1.214)-
-(-0.269,0.697)--(0.264,0.181)--(0.765,-0.331)--(1.232,-0.835)--(1.665,-1.329)
--(2.061,-1.811)--(2.421,-2.279)--(2.744,-2.73)--(3.032,-3.164); \draw(1.852,0
.765)--(1.96,0.524)--(2.072,0.248)--(2.187,-0.062)--(2.306,-0.408)--(2.427,-0.
788)--(2.549,-1.202)--(2.671,-1.649)--(2.793,-2.127)--(2.914,-2.633)--(3.032,-
3.164); \draw(-1.06,-2.136)--(-0.86,-1.938)
(-0.644,-1.738)--(-0.413,-1.535)--(-0.165,-1.331)
(0.099,-1.124)--(0.38,-0.917)--(0.676,-0.709)
(0.989,-0.501)--(1.317,-0.293)--(1.661,-0.087); \draw(1.168,-0.166)--(1.183,-0
.179)--(1.197,-0.193)--(1.211,-0.206)--(1.225,-0.22)--(1.238,-0.233)--(1.252,-
0.246)--(1.266,-0.26)--(1.279,-0.273)--(1.293,-0.286)--(1.307,-0.3); \draw(1.1
68,-0.166)--(1.199,-0.144)--(1.23,-0.122)--(1.261,-0.1)--(1.293,-0.079)--(1.32
4,-0.057)--(1.356,-0.035)--(1.387,-0.013)--(1.419,0.009)--(1.451,0.031)--(1.48
3,0.053); \draw(-1.421,1.729)--(-1.029,1.548)--(-0.657,1.367)--(-0.304,1.185)-
-(0.03,1.002)--(0.345,0.819)--(0.642,0.637)--(0.922,0.455)--(1.185,0.273)--(1.
431,0.093)--(1.661,-0.087)--(1.661,-0.087); \draw(1.852,0.765)--(1.841,0.728)-
-(1.828,0.682)--(1.814,0.626)--(1.798,0.559)--(1.78,0.482)--(1.76,0.393)--(1.7
38,0.292)--(1.715,0.179)--(1.689,0.053)--(1.661,-0.087); \draw(3.032,-3.164)--
(2.912,-2.851)--(2.787,-2.538)--(2.658,-2.224)--(2.524,-1.911)--(2.387,-1.599)
--(2.247,-1.29)--(2.104,-0.983)--(1.958,-0.679)--(1.81,-0.38)--(1.661,-0.087)-
-(1.661,-0.087); {\scalefont{0.5}\pgfputat{\pgfxy(1.993,-0.104)}{\pgfbox[cente
r,base]{$C_{n-1}$}}}
{\scalefont{0.5}\pgfputat{\pgfxy(0.751,-0.958)}{\pgfbox[center,base]{$C_n$}}}
\draw (0.751,-0.658) circle (0.02);
{\scalefont{0.5}\pgfputat{\pgfxy(-1.165,-2.35)}{\pgfbox[center,base]{$V_n$}}}
{\scalefont{0.5}\pgfputat{\pgfxy(-1.563,1.901)}{\pgfbox[center,base]{$V_0$}}}
{\scalefont{0.5}\pgfputat{\pgfxy(2.037,0.842)}{\pgfbox[center,base]{$V_j$}}} {
\scalefont{0.5}\pgfputat{\pgfxy(3.335,-3.48)}{\pgfbox[center,base]{$V_{n-1}$}}
}
\end{tikzpicture}
\caption{$\angle V_nC_{n-1}V_0$ is a right angle.}
\label{fig:right-angle-triangle-in-simplex}
\end{figure}
For the proof of Proposition~\ref{prop-about-distances-midpoint-vertices} we 
shall need the following lemma:

\begin{lemma} \label{Main-technical-lemma}
\label{fractions-identity}
For $n \in {\mathbb N}$ and  $\zeta > 1$, define $\beta,\gamma,\delta \in\mathbb{R}$ by
\begin{align*}
\cosh^2{\beta} &= \frac{n\zeta + 1}{n + 1},\\
\cosh^2{\gamma} &= \frac{\left(n + 1\right)\zeta^2}{n\zeta + 1},\\
\cosh^2{\delta} &= \frac{(n + 1)\zeta + 1}{n + 2}.
\end{align*}
Note that these definitions make sense since the quantities on the right-hand
side are all greater than one. Then the following identity holds: 
\begin{align} \label{main-trig-identity}
\cosh^2{\left(\gamma - \delta\right)}\cosh^2{\beta}=\cosh^2{\delta}.
\end{align}
\end{lemma}

The proof of Lemma~\ref{Main-technical-lemma}  is postponed to the Appendix.

\begin{proof}[{\bf Proof of Proposition~\ref{prop-about-distances-midpoint-vertices}}]
We argue by induction on the dimension $n$. For $n=1$, the simplex is a segment between two points, $V_0$ and $V_1$, distance $a$ apart. 
The point $C_1$ is their midpoint, and $C_0 = V_0$, since it is the midpoint
of a degenerate face containing only one point. It follows immediately that the point 
$C_1$ is on the segment $\left[C_0V_1\right]$,  and that
\begin{align*}
\cosh^2{d\left(V_0, C_1\right)} &= \cosh^2{d\left(V_1, C_1\right)} =
\cosh^2{\frac{a}{2}}= \frac{\cosh{a} + 1}{2},\\
\cosh^2{d\left(V_1, C_0\right)} &= \cosh^2{a}.
\end{align*}
Assume now that the proposition holds for $n=k$. Let $\triangle^{k+1}_a$ be a
regular $(k+1)$-simplex with vertices $\{V_0,\ldots,V_{k+1}\}$, and let
$F_{k+1}$ be its facet given by the convex hull of $\{V_0,\ldots,V_{k}\}$
(which is a regular $k$-simplex with side length $a$). Let $g$ be  the
geodesic line in ${\mathbb H}^{k+1}$ perpendicular to the facet $F_{k+1}$, and
passing through its midpoint $C_k$. The hyperplane in which $F_{k+1}$ lies
divides ${\mathbb H}^{k+1}$ into two half-spaces, and we denote by $\tilde{g}$
the part of $g$ that lies  on the same half-space as $\triangle^{k+1}_a$.

Next, let $P,Q$ be two points on $\tilde{g}$, such that $Q$ is between $C_k$ and $P$ and
\begin{align}
&\cosh^2{d\left(C_k, P\right)} = \frac{\left(k + 1\right)\cosh^2{a}}{k\cosh{a} + 1}\label{dist-ck-p},\\
&\cosh^2{d\left(Q, P\right)} = \frac{\left(k+1\right)\cosh{a} + 1}{k + 2}\label{dist-p-q}.
\end{align}

We remark  that for the above to be well defined the expression for the distance
between $C_k$ and $P$ must be larger than the expression for the distance
between $P$ and $Q$, and indeed, since $\cosh{a} >1$, one has that
\begin{align*}
\frac{\left(k+1\right)\cosh{a} + 1}{k + 2} < \frac{\left(k+1\right)\cosh{a}}{k + 1} = 
\frac{\left(k+1\right)\cosh^2{a}}{k\cosh{a} + \cosh{a}} < \frac{\left(k + 1\right)\cosh^2{a}}{k\cosh{a} + 1},
\end{align*}
and thus 
$d\left(Q, P\right) <d\left(C_k, P\right)$, as
the inverse hyperbolic cosine function is positive and  monotonically increasing.

From the definition of the points $P$ and $Q$ it follows that for every $0 \le j \le k$ one has $\angle P C_k V_j = \angle Q C_k V_j = \frac{\pi}{2}$, and thus using
the hyperbolic law of cosines we conclude that:
\begin{align}
  &\cosh^2{\dist{P}{V_j}} = \cosh^2{\dist{P}{C_k}}\cosh^2{\dist{C_k}{V_j}}\label{first-squared-identity},\\
  &\cosh^2{\dist{Q}{V_j}} = \cosh^2{\left(\dist{P}{C_k} - \dist{P}{Q}\right)}\cosh^2{\dist{C_k}{V_j}}\label{second-squared-identity}.
\end{align}
Using the induction hypothesis and~$(\ref{dist-ck-p})$ above,  equality  $(\ref{first-squared-identity})$ gives
\begin{align*}
  \cosh^2{\dist{P}{V_j}} =& \frac{\left(k + 1\right)\cosh^2{a}}{(k\cosh{a} + 1)} \cdot \frac{(k\cosh{a} + 1)}{(k + 1)} \ =\cosh^2{a}.
\end{align*}
This implies  that the distance from $P$ to each of the  vertices of $F_{k+1}$ equals $a$. Thus, from the uniqueness  (up to isometries) of the regular $(k+1)$-simplex and
our choice of the point $P$ we conclude that $P$ must coincide with $V_{k+1}$.

Next, we simplify the right-hand side of expression $(\ref{second-squared-identity})$ using Lemma~\ref{fractions-identity}, 
by replacing $n$ by $k$, $\zeta$ by $\cosh{a}$, $\beta$ by $\dist{C_k}{V_j}$, $\gamma$ by
$\dist{P}{C_k}$ and $\delta$ by $\dist{P}{Q}$. The lemma's premise holds by
the induction hypothesis and~$(\ref{dist-ck-p})$ and~$(\ref{dist-p-q})$ above, and we conclude that for every $0 \leq j \leq k$ one has
\begin{align} 
  \cosh^2{\dist{Q}{V_j}} = \cosh^2{\dist{Q}{P}} = \cosh^2{\dist{Q}{V_{k+1}}}.\label{q-is-eqidistant}
\end{align}
Thus, the distance between $Q$ and each vertex of $F_{k+1}$ 
equals the distance between $Q$ and $V_{k+1}$. 
Again, from the uniqueness property of the midpoint it follows that $Q$ must coincide with $C_{k+1}$,
and consequently assertion $(i)$ of the proposition holds for $n=k+1$.
Moreover, since the geodesic $g$ connecting $V_{k+1}$ and $C_k$ is perpendicular to the facet
$F_{k+1}$, it follows that $\angle C_{k+1} C_k V_j = \angle V_{k+1} C_k V_j = \frac{\pi}{2}$,
which proves assertion $(ii)$ for $n=k+1$.

Substituting $P=V_{k+1}$ and $Q=C_{k+1}$ in~$(\ref{dist-ck-p})$ and~$(\ref{dist-p-q})$ for $n=k$ yields
\begin{align*}
\cosh^2{d\left(C_{k+1}, V_{k+1}\right)} &= \frac{\left(k+1\right)\cosh{a} + 1}{k + 2},\\
\cosh^2{d\left(C_k, V_{k+1}\right)} &= \frac{\left(k + 1\right)\cosh^2{a}}{k\cosh{a} + 1}.
\end{align*}
The first equality together with~$(\ref{q-is-eqidistant})$ above prove assertion $(iii)$  for $n=k+1$,
and the second  $(iv)$. This completes the proof of  Proposition~\ref{prop-about-distances-midpoint-vertices}.
\end{proof}

We now turn to compute the centroid of the hyperbolic regular simplex.

\begin{proposition} \label{Proposition-centroid-of-simplex}
\label{hyperbolic-regular-simplex-mass}
With the above notations, let $\left(Z_n,z_n\right)$ be the centroid of n+1 point
masses of unit mass placed in the vertices of $\triangle_a^{n}$. Then,
\begin{align}
\label{centroid-of-vertices}
(Z_n,z_n) = \Bigl (C_n, \sqrt{ \left(n + 1\right)\left(n\cosh{a} + 1\right) } \Bigr)
\end{align}
\end{proposition}

\begin{proof}[{\bf Proof of Proposition~\ref{Proposition-centroid-of-simplex}}]

We start by showing that $Z_n$ coincides with $C_n$, the midpoint of $\triangle^n_a$.
Let $\tau$ be a permutation of 
$\{0,\ldots,n\}$, 
and let $\tilde{\sigma}$ be an isometry of $\mathbb{H}^n$ such that 
$\tilde{\sigma}V_j=V_{\tau j}$ for every $0 \le j \le n$. Note that
\[
(Z_n,z_n) :=  \centroid_{j=0}^{n}{(V_j, 1)} = \centroid_{j=0}^{n}(V_{\tau j}, 1) 
= \centroid_{j=0}^{n}(\tilde{\sigma} V_j, 1) = (\tilde{\sigma} Z_n,z_n).
\]
This implies that the point $Z_n$ lies at the same distance from $V_j$ and $V_{\tau j}$ since
\begin{equation} \label{eq-dist-to-vertices-prop} \dist{Z_n}{V_j}=\dist{\tilde{\sigma} Z_n}{\tilde{\sigma} V_j}=\dist{Z_n}{V_{\tau j}}. \end{equation}
Since~$(\ref{eq-dist-to-vertices-prop})$ holds for any $0 \leq j \leq n$ and  any permutation of the vertices, it follows that $Z_n$ is equidistant from all
the vertices of $\triangle^n_a$, and thus coincides with the midpoint of $\triangle^n_a$.

It remains to show that the masses in both sides of $(\ref{centroid-of-vertices})$ are indeed 
equal.  As before, we argue by induction on the dimension $n$.
For $n=1$ this follows immediatly from Definition~\ref{def-centroid}.
We assume the proposition holds for $n=k$.
From Definition~\ref{def-centroid} it follows that
\begin{align*}
z_{k+1} = \cosh{\dist{C_{k+1}}{V_{k+1}}} + z_k\cosh{\dist{C_{k+1}}{C_k}}.
\end{align*}
By combining this with the induction hypothesis we conclude that
\begin{align*}
z_{k+1} = \cosh{\dist{C_{k+1}}{V_{k+1}}} + 
\sqrt{ \left(k + 1\right)\left(k\cosh{a} + 1\right)}\cosh{\dist{C_{k+1}}{C_k}}.
\end{align*}
From assertion $(ii)$ of Proposition~\ref{hyperbolic-regular-simplex-measurements} it follows that  $\angle C_{k+1}C_kV_0 = \frac{\pi}{2}$, and 
thus using the hyperbolic law of cosines we obtain
\begin{align*}
z_{k+1} = \cosh{\dist{C_{k+1}}{V_{k+1}}} +
\sqrt{ \left(k + 1\right)\left(k\cosh{a} + 1\right)} \cdot
\frac{ \cosh{\dist{C_{k+1}}{V_{0}}}  }{\cosh{\dist{C_k}{V_0}}}.
\end{align*}
Since $ \cosh{\dist{C_{k+1}}{V_{0}}} =  \cosh{\dist{C_{k+1}}{V_{k+1}}}$, we further deduce that
\begin{align*}
z_{k+1} = \cosh{\dist{C_{k+1}}{V_0}}\left(1 + 
\frac{\sqrt{ \left(k + 1\right)\left(k\cosh{a} + 1\right)}}{\cosh{\dist{C_k}{V_0}}}\right).
\end{align*}
Finally, from assertion $(iii)$ and $(iv)$ of Proposition~\ref{hyperbolic-regular-simplex-measurements}  we get
\begin{align*}
z_{k+1} &=  \sqrt{ \frac{\left(k+1\right)\cosh{a} + 1}{k + 2}   }
  \Biggl(1 + \frac{\sqrt{ \left(k + 1\right)\left(k\cosh{a} + 1\right)}}
  {\sqrt{\frac{k\cosh{a} + 1}{k+1}}} \Biggr)\\
&= \sqrt{ \frac{\left(k+1\right)\cosh{a} + 1}{k + 2} }
  \left(1 + k + 1\right)
= \sqrt{ \left(k+2\right)\left(\left(k+1\right)\cosh{a} + 1\right)}.
\end{align*}
This completes the proof of Proposition~\ref{Proposition-centroid-of-simplex}.
\end{proof}
Recall that $\sigma_j$ stands for the reflection in ${\mathbb H}^n$ with respect
to the hyperplane in which the facet $F_j$ of the simplex $\triangle^n_a$ lies. 
The following proposition describes the centroid of a vertex of $\triangle^n_a$ and its reflection 
with respect to the opposite facet in terms of a weigthed center of mass of the other vertices.

\begin{proposition}
\label{corollary-combined-mass-equation}
Let  $\triangle^n_a$ be a hyperbolic regular $n$-simplex with edge length $a>0$ and 
vertices $\left\{V_0,\ldots,V_n\right\}$. 
Then for every  $0 \leq j \leq n$ one has
\begin{align*}
\left(V_j, 1\right)* \left(\sigma_jV_j, 1\right) = \centroid^{n}_{\substack{k=0\\k \neq j}} \biggl(V_k, \frac{2}{n-1 + \frac{1}{\cosh{a}}}\biggr).
\end{align*}
\end{proposition}
\begin{proof}[{\bf Proof of Proposition~\ref{corollary-combined-mass-equation}}]
Fix  $0 \leq j \leq n$. Since the facet $F_j$ is an $(n-1)$-regular simplex with edge length $a$, it follows from Proposition~\ref{hyperbolic-regular-simplex-mass} that
$$\centroid^{n}_{\substack{k=0\\k \neq j}} \left(V_k, 1 \right) = \Bigl( W_j, \sqrt{ n\left(\left(n-1\right)\cosh{a} + 1\right)} \Bigr),$$
where $W_j$ is the midpoint of  $F_j$. From Definition~$(\ref{linearity-of-centroid})$ it follows that one can 
multiply all the masses in the above equation by the constant $ \frac{2}{n-1 + \frac{1}{\cosh{a}}}$, and hence:
\begin{align} \label{centroid-reflection-rhs}
\centroid^{n}_{\substack{k=0\\k \neq j}} \left(V_k,\frac{2}{n-1 + \frac{1}{\cosh{a}}}\right) = \left(W_j, 2 \sqrt{ \frac{n\cosh^2{a}}{\left(n-1\right)\cosh{a} + 1}}\right)
\end{align}

On the other hand, the segment $[V_j, \sigma_j V_j]$ is invariant under
$\sigma_j$ and not fully contained by $F_j$, so must be perpendicular to $F_j$.
Assertion $(ii)$ of Proposition~\ref{hyperbolic-regular-simplex-measurements}, 
implies that this segment  must pass through $W_j$. Moreover, $\sigma_j$ is an
isometry that  leaves $W_j \in F_j$ invariant so $\dist{V_j}{W_j}=\dist{\sigma_jV_j}{W_j}$.
The point $W_j$ is therefore the midpoint between $V_j$ and $\sigma_j V_j$.

From Definition~\ref{def-centroid} it follows that
\begin{align} \label{centroid-of-reflection1}
\left(V_j, 1\right)* \left(\sigma_j V_j, 1\right) &= \bigl (W_j, \cosh{\dist{W_j}{V_j}}+\cosh{\dist{W_j}{\sigma_j V_j}}\bigr).
\end{align}
From assertion $(iii)$ of Proposition~\ref{hyperbolic-regular-simplex-measurements} it follows that both
$\cosh^2{\dist{W_j}{V_j}}$ and $\cosh^2{\dist{W_j}{\sigma_j V_j}}$ equal
$\frac{n\cosh^2{a}}{\left(n-1\right)\cosh{a} + 1}$, since $\dist{W_j}{V_j}$ and $\dist{W_j}{\sigma_j V_j}$ 
are distances between a vertex of a regular $n$-simplex and the midpoint of the facet in front of it.
Hence~$(\ref{centroid-of-reflection1}$) becomes
\begin{align} \label{centroid-of-reflection2}
\left(V_j, 1\right)* \left(\sigma_jV_j, 1\right) = \Biggl( W_j, 2 \sqrt{ \frac{n\cosh^2{a}}{\left(n-1\right)\cosh{a} + 1}}\Biggr).
\end{align}
The proof of the proposition now follows from~$(\ref{centroid-reflection-rhs}$)  and~$(\ref{centroid-of-reflection2}$).
\end{proof}

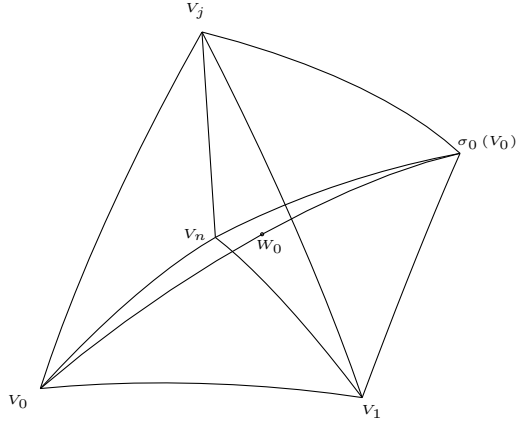
\begin{figure}[H]
\centering
\begin{tikzpicture}
\draw(2.602,1.073)--(2.264,1.004)--(1.923,0.924)--(1.583,0.833)--(1.246,0.732)
--(0.913,0.622)--(0.586,0.503)--(0.269,0.376)--(-0.039,0.242)--(-0.334,0.103)-
-(-0.616,-0.042); \draw(2.602,1.073)--(2.396,1.251)--(2.163,1.428)--(1.902,1.6
02)--(1.611,1.773)--(1.291,1.94)--(0.939,2.102)--(0.555,2.258)--(0.139,2.408)-
-(-0.31,2.55)--(-0.792,2.682); \draw(2.602,1.073)--(2.502,0.845)--(2.395,0.593
)--(2.28,0.317)--(2.158,0.018)--(2.029,-0.303)--(1.894,-0.644)--(1.754,-1.003)
--(1.611,-1.378)--(1.464,-1.766)--(1.315,-2.166); \draw(2.602,1.073)--(2.161,0
.956)--(1.674,0.789)--(1.147,0.573)--(0.586,0.309)--(0,0)--(-0.601,-0.352)--(-
1.205,-0.74)--(-1.8,-1.157)--(-2.375,-1.595)--(-2.918,-2.046);
\draw(-0.616,-0.042)--(-0.792,2.682); \draw(1.315,-2.166)--(1.17,-1.753)--(1.0
09,-1.32)--(0.833,-0.866)--(0.641,-0.395)--(0.434,0.092)--(0.213,0.593)--(-0.0
21,1.106)--(-0.268,1.627)--(-0.525,2.153)--(-0.792,2.682)--(-0.792,2.682);
\draw(1.315,-2.166)--(1.088,-1.868)--(0.866,-1.587)--(0.65,-1.324)--(0.441,-1.08)
--(0.241,-0.856)--(0.049,-0.652)--(-0.133,-0.469)--(-0.304,-0.306)--(-0.465,-0
.164)--(-0.616,-0.042); \draw(-0.616,-0.042)--(-0.802,-0.155)--(-1,-0.288)--(-
1.209,-0.44)--(-1.429,-0.613)--(-1.659,-0.805)--(-1.898,-1.016)--(-2.144,-1.24
7)--(-2.397,-1.496)--(-2.656,-1.763)--(-2.918,-2.046); \draw(-0.792,2.682)--(-
1.076,2.166)--(-1.346,1.653)--(-1.602,1.146)--(-1.842,0.646)--(-2.066,0.157)--
(-2.273,-0.318)--(-2.461,-0.778)--(-2.632,-1.22)--(-2.784,-1.643)--(-2.918,-2.
046); \draw(1.315,-2.166)--(0.905,-2.109)--(0.483,-2.061)--(0.053,-2.023)--(-0
.383,-1.995)--(-0.82,-1.978)--(-1.255,-1.971)--(-1.686,-1.974)--(-2.109,-1.988
)--(-2.52,-2.012)--(-2.918,-2.046); \draw (-0.001,0) circle (0.02);
{\scalefont{0.5}\pgfputat{\pgfxy(0.099,-0.2)}{\pgfbox[center,base]{$W_0$}}}
{\scalefont{0.5}\pgfputat{\pgfxy(2.962,1.181)}{\pgfbox[center,base]{$\sigma_0\left(V_0\right)$}}}
{\scalefont{0.5}\pgfputat{\pgfxy(-0.877,-0.046)}{\pgfbox[center,base]{$V_n$}}}
{\scalefont{0.5}\pgfputat{\pgfxy(-0.871,2.95)}{\pgfbox[center,base]{$V_j$}}}
{\scalefont{0.5}\pgfputat{\pgfxy(1.447,-2.382)}{\pgfbox[center,base]{$V_1$}}}
{\scalefont{0.5}\pgfputat{\pgfxy(-3.209,-2.251)}{\pgfbox[center,base]{$V_0$}}}
\end{tikzpicture}
\caption{The center of mass of a vertex and its reflection is positioned 
 in the center of mass of the rest of the vertices}
\label{fig:middle-fo-point-and-its-reflection}
\end{figure}

We finish this subsection with the following simple observation:

\begin{lemma} \label{tech-lemma}
\label{exterior-point-iff-zero-mass}
Let $\triangle_{a}^n$ be a regular $n$-simplex with vertices $\left\{V_0,\ldots,V_n\right\}$,
and $\left\{w_0,\ldots,w_n\right\}$ be non-negative real numbers. Let $P=\centroid_{j=0}^n\left(V_j,w_j\right)$
be the centroid of point masses placed in the vertices $\left\{V_0,\ldots,V_n\right\}$.
Then,
 $$P \in F_k \Leftrightarrow w_k=0.$$
\end{lemma}

\begin{proof}[{\bf Proof of Lemma~\ref{tech-lemma}}]
The commutativity and associativity  of the $\Conv$ operator allow one to  write
\begin{align*}
P=\left(V_k, w_k\right)* \left(Q,m\right),
\end{align*}
for some $m\ge0$, where
\begin{align*}
Q:=\centroid_{\substack{j=0\\j\neq k}}^{n}\left(V_j,w_j\right)\in F_k
\end{align*}
Moreover, $V_k \not\in F_k$, and the geodesic through $V_k$ and $Q$ intersects  $F_k$
at one point at most, namely $Q$.
Since $P$ lies on the above mentioned geodesic connecting $V_k$ and $Q$, 
$$P \in F_k \Leftrightarrow P \mbox{ coincides with } Q \Leftrightarrow w_k=0.$$
\end{proof}
\subsection{Billiards in regular hyperbolic simplices}
Billiard dynamics in hyperbolic space, and in particular polygonal and
polyhedral billiards, have been extensively studied both in the context of
mathematics and physics  (see e.g.,~\cite{DHN,Min,Ves}). As in the Euclidean
case, when defining billiard dynamics in the hyperbolic space for a non-smooth
domain, one  runs into certain technical difficulties in describing the dynamics
at the singular parts of the boundary. To avoid these difficulties, in what
follows we shall consider only closed billiard orbits in $\triangle^n_a \subset
{\mathbb H}^n$ which bounce at the interior of the facets. More precisely,

\begin{definition} \label{def-billiard-traj}
A closed billiard orbit of period $n+1$ inside
$\triangle^n_a \subset {\mathbb H}^n$
is a closed polygonal curve consisting of geodesic segments, and specified by
a sequence of points  $\left\{P_j\right\}_{j=0}^n \in \partial \triangle^n_a$
such that:
\begin{itemize}
  \item[(i)] For every $0 \le j \le n$, there is a unique $0 \le k \le n$ such that  $P_j \in {F}_k$, \label{bouncing-point-on-unique-facet}
  \item[(ii)] $P_j \in [P_{j-1}, \sigma_j \left( P_{j+1} \right)]$, where $\sigma_j$ is the reflection in ${\mathbb H}^n$ with respect
to the  hyperplane in which the facet $F_j$ lies.  \label{reflection-property}
\end{itemize}
\end{definition}

\noindent{\bf Remark:} {\rm It is not hard to check that the above definition
is equivalent to the definition of billiard trajectories as critical points of
the length functional, where trajectories  passing through non-smooth parts of
the boundary of $\triangle^n_a$ are excluded (see e.g.~\cite{KT}).}

\section{Proof of the Main Theorem} \label{section-proof-main-theorem}

Let $\triangle^n_a$ be a regular simplex in the hyperbolic space ${\mathbb
H}^n$ with side length $a$.  We turn now to the construction of a billiard
trajectory inside $\triangle^n_a$ which bounces at the interior of any facet
exactly once.  The (bouncing) points of this orbit will be described as
locations of centroids of masses positioned at the vertices of the simplex
$\triangle^n_a$.  For this end, let us first define a finite sequence of real
non-negative numbers that will serve as a pool from which these masses will be
drawen.

\begin{lemma} \label{mass-sequence-existance}
For every $2<n\in\mathbb{N}$, and $0<a\in\mathbb{R}$, there exist
$0<\lambda_{n,a}\in\mathbb{R}$ and a sequence of $n + 2$ real numbers
$\alpha_0,\ldots,\alpha_{n+1}$, such that 
\begin{equation} \label{positivity-of-the-seq}
 \alpha_0 = \alpha_{n+1} = 0, \
\alpha_1 = \alpha_n = 1, \    \alpha_j > 0 \  {\rm for \  every \ }  1 < j <
n, \end{equation} and for every $1 \le j \le n$ one has
\begin{align}
\lambda\alpha_{j} = \alpha_{j-1} + \alpha_{j+1} +  \frac{2}{n-1 + \frac{1}{\cosh{a}}}\label{weight-equation}
\end{align}
\end{lemma}
The proof of Lemma~\ref{mass-sequence-existance} is postponed to the Appendix.

In what follows, to ease notations, let us extend the vertices $\{V_j\}_{j=1}^n$ to all $j\in\mathbb{Z}$ by
cyclicly repeating  them in both directions, thus creating an $(n+1)$-periodic
sequence $\left\{V_j\right\}_{j \in \mathbb{Z}}$ such that $V_j=V_k$
if $j \equiv k \mod{n+1}$.

With Lemma~\ref{mass-sequence-existance} at our disposal, we now
define the bouncing points of the  billiard 
trajectory as locations of centroids of point masses placed in the vertices of the simplex
of $\triangle^n_a$:

\begin{definition}
\label{billiard-definition}
With the above notations, for every $0 \leq j \leq n$, let
\begin{align} \label{def-of-bouncing-points}
\left(P_j,m_j\right) &:= \centroid^{n}_{k=0} \left(V_{k+j}, \alpha_k\right),
\end{align}
where $\{ \alpha_k\}_{k=0}^{n}$ is a sequence which satisfies properties~$(\ref{positivity-of-the-seq})$  and~$(\ref{weight-equation})$, 
whose existence is ensured by Lemma~\ref{mass-sequence-existance} above.
\end{definition}

Finally, we are now in position to prove our main result.
\begin{proof}[{\bf Proof of Theorem~\ref{main-thm}}]
We will show that the sequence of points $\{P_j\}_{j=0}^n$, defined in equation~$(\ref{def-of-bouncing-points})$ form a periodic billiard trajectory in $\triangle^n_a$.
%
%

First, note that in the definition of the point $P_j$, the mass positioned in the vertex $V_j$ is
$\alpha_0=0$. Thus, from Lemma~\ref{exterior-point-iff-zero-mass} it follows that $P_j$ belongs 
to the facet $F_j$. On the other hand, the rest of the masses that appear in the definition of the point $P_j$ are strictly positive.
Hence, applying Lemma~\ref{exterior-point-iff-zero-mass}  once again, this time to the facet $F_j$ 
(considered as regular $(n-1)$-simplex), we obtain that the point $P_j$  does not belong to any other facet of the simplex, as required
by property $(i)$ of Definition~\ref{def-billiard-traj}.

It remains to show that the sequence $\{P_j\}_{j=0}^n$ satisfies property $(ii)$ of 
Definition~\ref{def-billiard-traj}, i.e., that $P_j\in[P_{j-1}, \sigma_j P_{j+1}]$, for every $0 \leq j \leq n$
(see Figure \ref{fig:strategy-of-main-thm-proof} below).
In fact we will show, by means of center of mass arguments, that
\begin{align} 
(P_j, \lambda m_j) = \left(P_{j-1},m_{j-1}\right) \centroid \left(\sigma_j P_{j+1},m_{j+1}\right), \label{bouncing-point-center-of-mass-condition}
\end{align}
where $\lambda>0$ is the positive constant ensured by Lemma~\ref{mass-sequence-existance}, and thus 
$P_j\in[P_{j-1}, \sigma_j P_{j+1}]$ as required.

\begin{figure}[H]
\centering
\begin{tikzpicture}[x=7cm,y=7cm]
\draw [thick] (0,1.349) -- (0, 0.74);
\draw [thick] (0,0.74) arc (122.849:90.435:0.881);
\draw [thick] (0,1.349) arc (57.15:33.284:1.606);
\draw (-0.472,0.881) arc (146.715:122.849:1.606);
\draw (-0.472,0.881) arc (89.564:57.15:0.881);
\draw [thin,gray,-latex] (0,1) arc (119.655:104.292:1.15);
\draw [thin,gray] (-0.286,1.115) arc (75.707:47.007:1.15);
{\scalefont{0.7}\pgfputat{\pgfxy(0,0.7)}{\pgfbox[center,base]{$V_2$}}};
{\scalefont{0.7}\pgfputat{\pgfxy(0,1.37)}{\pgfbox[center,base]{$V_0$}}};
{\scalefont{0.7}\pgfputat{\pgfxy(0.491,0.881)}{\pgfbox[center,base]{$V_1$}}};
{\scalefont{0.7}\pgfputat{\pgfxy(-0.52,0.881)}{\pgfbox[center,base]{$\sigma_1 V_1$}}};
{\scalefont{0.7}\pgfputat{\pgfxy(-0.03,0.97)}{\pgfbox[center,base]{$P_1$}}};
{\scalefont{0.7}\pgfputat{\pgfxy(0.255,0.811)}{\pgfbox[center,base]{$P_0$}}};
{\scalefont{0.7}\pgfputat{\pgfxy(0.32,1.115)}{\pgfbox[center,base]{$P_2$}}};
{\scalefont{0.7}\pgfputat{\pgfxy(-0.35,1.115)}{\pgfbox[center,base]{$\sigma_1 P_2$}}};
\end{tikzpicture}
\caption{Property $(ii)$ of Definition~\ref{def-billiard-traj} for $n=2,j=1$: $P_1\in[P_0, \sigma_1 P_2]$}
\label{fig:strategy-of-main-thm-proof}
\end{figure}

To this end, we start with the following computation. Let $0 \le j \le n$.
From the properties of the center of mass (Definition~\ref{def-centroid}),
the choice of the sequence $\{\alpha_k\}_{k=0}^{n+1}$
(Definition~\ref{mass-sequence-existance}), and the definition of the
bouncing points $\{P_j\}_{j=0}^n$ (Definition~\ref{billiard-definition}) 
it follows that
\begin{align*}
\left(\sigma_j\left(P_{j+1}\right), m_{j+1}\right)
  &= \centroid^{n}_{k=0} \left(\sigma_j\left(V_{k+j+1}\right), \alpha_k\right) \\
  &= \centroid^{n-1}_{k=0} \left(\sigma_j\left(V_{k+j+1}\right), \alpha_k\right) \Conv \left(\sigma_j \left(V_{j+n+1}\right), \alpha_n\right) \\
  &= \centroid^{n-1}_{k=0} \left(\sigma_j\left(V_{k+j+1}\right), \alpha_k\right) \Conv \left(\sigma_j \left(V_{j+n+1}\right), 1\right) \\
  &= \centroid^{n-1}_{k=0} \left(\sigma_j\left(V_{k+j+1}\right), \alpha_k\right) \Conv \left(\sigma_j V_j, 1\right)\\
  &= \centroid^{n}_{k=1} \left(\sigma_j\left(V_{k+j}\right), \alpha_{k-1}\right) \Conv \left(\sigma_j V_j, 1\right)\\
  &= \centroid^{n}_{k=1} \left(V_{k+j}, \alpha_{k-1}\right) \Conv \left(\sigma_j V_j, 1\right),
\end{align*}
where the last equality holds since $V_{k+j}$ does not equal $V_{j}$ for $1 \le k \le n$, 
and is therfore a vertex of $F_j$ which is invariant under the reflection $\sigma_j$.

On the other hand, 
\begin{align*}
\left(P_{j-1}, m_{j-1}\right)
  &= \centroid^{n}_{k=0} \left(V_{k+j-1}, \alpha_k\right) \\
  &= \centroid^{n}_{k=2} \left(V_{k+j-1}, \alpha_k\right)  \Conv \left(V_{j-1}, \alpha_0\right) 
  * \left(V_j, \alpha_1\right)\\  
  &= \centroid^{n}_{k=2} \left(V_{k+j-1}, \alpha_k\right)  \Conv \left(V_{n+j}, \alpha_{n+1}\right) 
  * \left(V_j, \alpha_1\right)\\
  &= \centroid^{n+1}_{k=2} \left(V_{k+j-1}, \alpha_k\right)  \Conv \left(V_j, \alpha_1\right)\\
  &= \centroid^{n}_{k=1} \left(V_{k+j}, \alpha_{k+1}\right)  \Conv \left(V_j, 1\right).
\end{align*}
Using once again the properties of the centroid, a direct calculation of the center of mass of the two point masses above yields
that the right-hand side of relation~$(\ref{bouncing-point-center-of-mass-condition})$ equals
\begin{align*}
   &  \Bigl( \centroid^{n}_{k=1} \left(V_{k+j}, \alpha_{k-1}\right)\Bigl) \Conv \left(\sigma_j V_j, 1\right)
      \Conv 
      \Bigl(\centroid^{n}_{k=1} \left(V_{k+j}, \alpha_{k+1}\right)\Bigl) \Conv \left( V_j, 1\right)=\\
   & \Bigl(\centroid^{n}_{k=1} \left(V_{k+j}, \alpha_{k - 1}  + \alpha_{k + 1} \right)\Bigl) \Conv \left(V_j, 1\right)\Conv \left(\sigma_jV_j, 1\right).
\end{align*}
Using Proposition~\ref{corollary-combined-mass-equation} above, one can replace $\left(V_j, 1\right)* \left(\sigma_jV_j, 1\right)$
by an expression free of the reflection $\sigma_j$ and thus obtain:
\begin{align*}
&\left(\sigma_jP_{j+1}, m_{j+1}\right) \Conv \left(P_{j-1}, m_{j-1}\right)\\
    &=\Bigl(\centroid^{n}_{k=1} \left(V_{k+j},  \alpha_{k - 1}  + \alpha_{k + 1} \right) \Bigl)\Conv 
      \left(\centroid^{n}_{k=1} \left(V_{k+j}, \frac{2}{n-1 + \frac{1}{\cosh{a}}}\right)\right)\\
    &=\centroid^{n}_{k=1}\left(V_{k+j},  \alpha_{k - 1}  + \alpha_{k + 1} +
     \frac{2}{n-1 + \frac{1}{\cosh{a}}}\right).
\end{align*}
Since equation~$(\ref{weight-equation})$ is satisfied for every $1 \le k \le n$, we conclude that: 
$$
   \left(\sigma_jP_{j+1}, m_{j+1}\right) * \left(P_{j-1}, m_{j-1}\right) = \centroid^{n}_{k=1} \left(V_{k+j}, \lambda\alpha_{k}\right).
$$

The expression on the right-hand side is very similar to the definition of $\left(P_0, \lambda m_0\right)$,
the only difference being in the range of $k$'s -- it does not include $0$. However, 
since $0 = \alpha_0 = \lambda\alpha_0$, extending the range of $k$ to include $0$ has no effect on the
value of the expression and hence: 
\begin{align*}
   \left(\sigma_jP_{j+1}, m_{j+1}\right) * \left(P_{j-1}, m_{j-1}\right)= \centroid^{n}_{k=0} \left(V_{k+j}, \lambda\alpha_{k}\right)= \left(P_0, \lambda m_0\right)
\end{align*}
This completes the proof of Theorem~\ref{main-thm}.
\end{proof}
 
\noindent {\bf APPENDIX}

\begin{proof}[{\bf Proof of Lemma~\ref{Main-technical-lemma}}]
A direct computation shows that
\begin{align*}
&\sinh^2{\gamma} = \cosh^2{\gamma}-1=\frac{\left(n + 1\right)\zeta^2}{n\zeta + 1}-1 =
\frac{\left(n\zeta + \zeta + 1\right)\left(\zeta-1\right)}{n\zeta + 1},\\
&\sinh^2{\delta} = \cosh^2{\delta}-1=\frac{\left(n+1\right)\zeta + 1}{n + 2}-1 =
\frac{\left(n+1\right)\left(\zeta - 1\right)}{n + 2}.
\end{align*}
Combining this with a well known hyperbolic trigonometric identity gives 
\begin{align*}
&\cosh^2{\left(\gamma - \delta\right)}=\Bigl ( \cosh(\gamma)\cosh(\delta)-\sinh(\gamma)\sinh(\delta)  \Bigr)^2 \\ 
&=\cosh^2{\gamma}\cosh^2{\delta}+\sinh^2{\gamma}\sinh^2{\delta} -2\sqrt{\cosh^2{\gamma}\cosh^2{\delta}\sinh^2{\gamma}\sinh^2{\delta}}\\ 
&=\frac{\left(n + 1\right)\zeta^2}{(n\zeta + 1)}\cdot \frac{(n\zeta + \zeta + 1)}{(n + 2)} +  \frac{\left(n\zeta + \zeta + 1\right)\left(\zeta-1\right)}{(n\zeta + 1)} \cdot  \frac{\left(n+1\right)\left(\zeta - 1\right)}{(n + 2)}  \\
& \ \ \ -2\sqrt{\frac{\left(n + 1\right)\zeta^2}{(n\zeta + 1)}\cdot \frac{(n\zeta + \zeta + 1)}{(n + 2)} \cdot  \frac{\left(n\zeta + \zeta + 1\right)\left(\zeta-1\right)}{(n\zeta + 1)} \cdot  \frac{\left(n+1\right)\left(\zeta - 1\right)}{(n + 2)} }\\
&=\frac{ \left(n + 1\right)\left(n\zeta+\zeta + 1\right)}{\left(n\zeta+1\right)\left(n+2\right)}\Bigl(
\zeta^2 + \left(\zeta-1\right)^2-2\zeta\left(\zeta-1\right)\Bigr)\\
&=\frac{ \left(n + 1\right)\left(n\zeta+\zeta + 1\right)}{\left(n\zeta+1\right)\left(n+2\right)}.
\end{align*}
Thus, we conclude that
\begin{align*}
&\cosh^2{\beta}\cosh^2{\left(\gamma - \delta\right)}=  \frac{(n\zeta + 1)}{(n + 1)}  \frac{ \left(n + 1\right)\left(n\zeta+\zeta + 1\right)}{\left(n\zeta+1\right)\left(n+2\right)}=\frac{(n + 1)\zeta + 1}{n + 2}=\cosh^2{\delta}
\end{align*}
This completes the proof of Lemma~\ref{Main-technical-lemma}.
\end{proof}

\begin{proof}[{\bf Proof of Lemma~\ref{mass-sequence-existance}}]
\label{proof-of-mass-sequence-existance}

Define $h:\left[0,\infty\right)\rightarrow\mathbb{R}_+$ by
\begin{align*}
h\left(x\right):=\frac{x^{\frac{n-1}{2}} + x^{-\frac{n-1}{2}}}{x^{\frac{n+1}{2}} + x^{-\frac{n+1}{2}}}
\end{align*}
It can be directly checked that $h$ is differentiable on $[0,\infty)$, and that: 
\begin{align*}
  \label{h-result}
   h'(x)=
  \left(x - x^{-1}\right)  \cdot
	   \frac{
	    -\sum\limits_{j=1}^{n}x^{n-2j}
	    -nx^{-1}
	   }{
	     \left(x^{\frac{n+1}{2}} + x^{-\frac{n+1}{2}}\right)^2
	   },
\end{align*}
for $x>1$ and that $h_{+}'(0)=1$. Next, define $g:\left[1,\infty\right)\rightarrow\mathbb{R}$ by
\begin{align*}
g\left(y\right):=h\left(y+\sqrt{ y^2 - 1 }\right) - 1 + \left(y-1\right)\left(n-1+\frac{1}{\cosh{a}}\right).
\end{align*}
The function $g$ is also 
differentiablel, and for $y>1$ one has
\begin{align*} 
g'(y)=h'\left(y+\sqrt{ y^2 - 1 }\right) \cdot
\biggl(1+\frac{y}{\sqrt{ y^2 - 1 }}\biggr)+\left(n-1+\frac{1}{\cosh{a}}\right).
\end{align*}
Set $\xi_y = y+\sqrt{ y^2 - 1 }$,  and note that $\xi_y^{-1}=y-\sqrt{ y^2 - 1 }$.
The derivative  $g'(y)$   
can be now expressed by means of $\xi_y$ as follows:
\begin{align*}
g'\left(y\right) &=h'\left(\xi_y\right) \cdot \left(1+\frac{\xi_y+\xi_y^{-1} }{\xi_y-\xi_y^{-1} }\right)+\left(n-1+\frac{1}{\cosh{a}}\right) \\
& =\left(\xi_y - \xi_y^{-1}\right)  \cdot
	   \frac{
 \Bigl(    -\sum\limits_{j=1}^{n}\xi_y^{n-2j}
	    -n\xi_y^{-1} \Bigr) 
	   }{
 	     \left(\xi_y^{\frac{n+1}{2}} + \xi_y^{-\frac{n+1}{2}}\right)^2
	   }
\cdot \left(1+\frac{\xi_y+\xi_y^{-1}}{\xi_y-\xi_y^{-1}}\right)+\left(n-1+\frac{1}{\cosh{a}} \right).
\end{align*}
Thus, we conclude that
\begin{align}
g'(y) =   \frac{
	    -\sum\limits_{j=1}^{n}\xi_y^{n-2j}
	    -n\xi_y^{-1}
	   }{
	     \left(\xi_y^{\frac{n+1}{2}} + \xi_y^{-\frac{n+1}{2}}\right)^2
	   } \cdot 2\xi_y+\left(n-1+\frac{1}{\cosh{a}} \right).
\end{align}
The function $g(y)$ is smooth and defined for $y=1$, and hence
\begin{align*}
g'(1)&=\lim_{y \to 1^+}g'\left(y\right) =2\frac{
	    -\sum\limits_{j=1}^n 1
	    -n1^{-1}
	   }{
	     \left(1^{\frac{n+1}{2}} + 1^{-\frac{n+1}{2}}\right)^2
	   }+n-1+\frac{1}{\cosh{a}}=
	   \\& = 2\frac{
	    -n
	    -n
	   }{
	     \left(1 + 1\right)^2
	   }+n-1+\frac{1}{\cosh{a}}=
	   \frac{1}{\cosh{a}}-1<0.
\end{align*}
This implies in particular that  there exists $\epsilon>0$ for which
$g(1+\epsilon)<g(1)=0$. On the other hand, it is not hard to verify that for
large enough values of $y$, one has $g(y)>0$. The intermediate value theorem
implies that there exists $1+\epsilon < y_0 $ such that $g(y_0)=0$. Next
define:
\begin{align*}
\lambda&:=2y_0>2,\\
\xi&:=\xi_{y_0}>1,\\
b&:=\frac{2}{\left(n-1+\frac{1}{\cosh{a}}\right)}.
\end{align*}
Note that with these notations, the equation
$$
0=g(y_0)=h\left(y_0+\sqrt{y_0^2-1}\right)-1+\left(y_0-1\right)\left(n-1+\frac{1}{\cosh{a}}\right),
$$
can be rewritten as
$$
h\left(\xi\right)=1-\left(\frac{\lambda}{2}-1\right)\frac{2}{b}=1-\frac{\lambda-2}{b}.
$$
We are now in a position to provide an explicit formula for a sequence
$\left\{\alpha_j\right\}_{j=0}^{n+1}$ that satisfies the required conditions
of the lemma. Let,
\begin{align*}
\alpha_j &:= \frac{b}{\lambda - 2}\left(1 - \frac{\xi^{j-\frac{n+1}{2}} + \xi^{\frac{n+1}{2} - j}}{ \xi^\frac{n+1}{2} + \xi^{-\frac{n+1}{2}}}\right).
\end{align*}
Note that by definition 
\begin{align*}
\alpha_0 = \alpha_{n+1} &= \frac{b}{\lambda - 2}\left(1 - \frac{\xi^{-\frac{n+1}{2}} + \xi^{\frac{n+1}{2}}}{ \xi^\frac{n+1}{2} + \xi^{-\frac{n+1}{2}}}\right) = 0,
\end{align*}
and
\begin{align*}
\alpha_1 = \alpha_n &= \frac{b}{\lambda - 2}\left(1 - \frac{\xi^{1-\frac{n+1}{2}} + \xi^{\frac{n+1}{2}-1}}{ \xi^\frac{n+1}{2} + \xi^{-\frac{n+1}{2}}}\right)
= \frac{b}{\lambda - 2}\Bigl(1- h\left(\xi\right)\Bigr)\\ &= \frac{b}{\lambda - 2}\cdot\frac{\lambda - 2}{b}=1.
\end{align*}
Next, let us verify that the sequence $\alpha_j$ satisfies $(\ref{weight-equation})$ for $1 \le j \le n$. Indeed,
\begin{align*}
\alpha_{j-1} + \alpha_{j+1} &= \frac{b}{\lambda - 2}\left(1 - \frac{\xi^{j-1-\frac{n+1}{2}} + \xi^{\frac{n+1}{2} - j + 1}}{ \xi^\frac{n+1}{2} +\xi^{-\frac{n+1}{2}} }\right) \\
&+ \frac{b}{\lambda - 2}\left(1 - \frac{\xi^{j+1-\frac{n+1}{2}} + \xi^{\frac{n+1}{2} - j-1}  }{ \xi^\frac{n+1}{2} +\xi^{-\frac{n+1}{2}} }\right),
\end{align*}
which can be simplified to:
\begin{align*}
\alpha_{j-1} + \alpha_{j+1} &= \frac{b}{\lambda - 2}\left(\lambda - \frac{\left(\xi^{j-\frac{n+1}{2}} + \xi^{\frac{n+1}{2}-j}   \right)  }{\left( \xi^\frac{n+1}{2} + \xi^{-\frac{n+1}{2}} \right)}\left( \xi + \xi^{-1}    \right)\right) - b.
\end{align*}
Now, by substituting $\lambda$ for $\xi+\xi^{-1}$, we obtain, as required, that:    
\begin{align*}
\alpha_{j-1} + \alpha_{j+1} &= \lambda \frac{b}{\lambda - 2}\left(1 - \frac{\xi^{j-\frac{n+1}{2}} + \xi^{\frac{n+1}{2}-j}}{ \xi^\frac{n+1}{2} + \xi^{-\frac{n+1}{2}} }\right) - b = \lambda\alpha_j- b.
\end{align*}
Finally, to complete the proof of the lemma, it remains to show that  $\alpha_j > 0$ for every  $0 < j < n+1$.
For this end, rewrite $\alpha_j$ as:
\begin{align} \label{formula-for-alphaj}
\alpha_j = \frac{b}{\lambda - 2} - \frac{b}{\left(\lambda - 2\right)\left( \xi^\frac{n+1}{2} + \xi^{-\frac{n+1}{2}} \right)}\left( \xi^{j-\frac{n+1}{2}} + \xi^{\frac{n+1}{2} - j}     \right).
\end{align}
Note that  the expression  $\xi^t + \xi^{-t}$ is monotonically increasing with respect to $\left|t\right|$, and hence the expression
$(\xi^{j-\frac{n+1}{2}} + \xi^{\frac{n+1}{2} - j}) $
is monotonically increasing with respect to $\left|j - \frac{n+1}{2}\right|$.
Moreover, recall that $\lambda > 2$ and $b > 0$.
Therefore, the expression on the right hand side of~$(\ref{formula-for-alphaj})$
achieves a global maxima at $j=\frac{n+1}{2}$, and  strictly
decreases as $\left|j - \frac{n+1}{2}\right|$ increases, that is -- as $j$
approaches $0$ on one side,  or $n+1$ on the other. Since it was already shown
that $\alpha_0 = \alpha_{n+1} = 0$, we  conclude that $\alpha_j > 0$ for all
$0 < j < n+1$. This completes the proof of Lemma~\ref{mass-sequence-existance}.
\end{proof}

\noindent Oded Badt\\
\noindent School of Mathematical Science, Tel Aviv University, Tel Aviv, Israel\\
\noindent {\it e-mail}: odedbadt@post.tau.ac.il
\bigskip

\noindent Yaron Ostrover\\
\noindent School of Mathematical Science, Tel Aviv University, Tel Aviv, Israel\\
\noindent {\it e-mail}: ostrover@post.tau.ac.il

\end{document}